\documentclass{amsart}
\usepackage{fullpage}
\usepackage{amsfonts,amsmath,amscd,amssymb}
\usepackage{dsfont,mathtools}
\usepackage{pgf,tikz,tikz-cd}
\usetikzlibrary{arrows,positioning}
\usepackage{xspace}
\usepackage{stmaryrd}
\usepackage{latexsym}
\usepackage[all]{xy}

\newtheorem{theorem}{Theorem}[section]

\newtheorem{prop}[theorem]{Proposition}

\newtheorem{cor}[theorem]{Corollary}

\newcommand{\Aut}{{\mathrm A}{\mathrm u}{\mathrm t}}

\newcommand{\Opext}{{\mathrm O}{\mathrm p}{\mathrm e}{\mathrm x}{\mathrm t}}

\newcommand{\Ass}{{\mathrm A}{\mathrm s}{\mathrm s}}

\newcommand{\sS}{\mathcal{S}}

\newcommand{\sG}{\mathcal{G}}
\newcommand{\sI}{\mathcal{I}}
\newcommand{\sH}{\mathcal{H}}

\newcommand{\vi}{\ensuremath{\mathbf{i}}\xspace}

\newcommand{\vs}{\ensuremath{\mathbf{s}}\xspace}
\newcommand{\vt}{\ensuremath{\mathbf{t}}\xspace}

\newcommand{\fourpartdef}[8]{\left\{
	\begin{array}{ll}
		#1 & \mbox{if } #2 \\
		#3 & \mbox{if } #4 \\
		#5 & \mbox{if } #6 \\
		#7 & \mbox{if } #8
	\end{array}
	\right.}

\begin{document}



\title[Covering Groups]{Covering Groups of Nonconnected Topological Groups and 2-Groups}
\author{Dmitriy Rumynin}
\email{D.Rumynin@warwick.ac.uk}
\address{Department of Mathematics, University of Warwick, Coventry, CV4 7AL, UK\newline
\hspace*{0.31cm}  Associated member of Laboratory of Algebraic Geometry, National
Research University Higher School of Economics, Russia}
\thanks{The research was partially supported by the Russian Academic Excellence Project `5--100' and by the Max Planck Society.
\newline \hspace*{0.31cm} Word count: 5129 words (with numbers), 4581 words (without numbers)}
\author{Demyan Vakhrameev}
\email{dem.vakh@hotmail.co.uk}
\address{Department of Mathematics, University of Warwick, Coventry, CV4 7AL, UK}
\author{Matthew Westaway} 
\email{M.P.Westaway@warwick.ac.uk}
\address{Department of Mathematics, University of Warwick, Coventry, CV4 7AL, UK}

\date{March 26, 2019}

\begin{abstract}
  We investigate the universal cover of a topological group that is not
  necessarily connected.
  Its existence as a topological group is governed by a Taylor cocycle,
  an obstruction in 3-cohomology.
  Alternatively, it always exists as a topological 2-group.
The splitness of this 2-group is also governed by an obstruction in 3-cohomology, a Sinh cocycle.
  We give explicit formulas for both obstructions and show that they
  are equal.
\end{abstract}

\maketitle

Let $G$ be a locally arcwise connected,
semilocally simply-connected topological group 
(e.g., a Lie group), $\pi:\widetilde{G}\rightarrow G$
a universal cover of the underlying space of $G$.
If $G$ is connected, a choice of a point in $\pi^{-1} (1_G)$
supplies $\widetilde{G}$ with a topological group structure
so that $\pi$ is a homomorphism of topological groups. 
If $G$ is not connected, we must specify multiplication of paths
(or loops)
living on different connected components.

This problem was investigated by Taylor over 60 years ago \cite{Ta2}.
The conditions on the topological group $G$ ensure
existence of the identity component $G_0$, the component group
$\pi_0 (G) = G/G_0$ and the universal cover of the identity
component $\widetilde{G}_0$ with the abelian fundamental
group $\pi_1 (G)$.
Existence of a group structure on $\widetilde{G}$ is controlled by the Taylor cocycle
$\dot\eta^\sharp\in Z^3(\pi_0(G),\pi_1(G))$.
Its cohomology class 
$[\dot\eta^\sharp]\in H^3(\pi_0(G),\pi_1(G))$ 
is an obstruction for existence of the universal cover group 
\cite[Theorem (6.5)]{Ta2} (cf. Theorem~\ref{main_th}):
\begin{center}
{\em The topological group $\widetilde{G}$ exists if and only if $[\dot\eta^\sharp]=0$.}
\end{center}




The question of uniqueness of $\widetilde{G}$ 
is subtle because 
we have two notions of uniqueness to consider.
Two group structures on the topological space $\widetilde{G}$
are equivalent if they define congruent \cite[p. 64]{ML}
extensions, but there are two extensions in the frame:
$$
1 \to \pi_1(G) \to \widetilde{G} \to G \to 1
\ \ \ \mbox{ and } \ \ \ 
1 \rightarrow \widetilde{G}_{0} \rightarrow
\widetilde{G}\rightarrow \pi_0(G)
\rightarrow 1.
$$
We consider uniqueness of the first, but only examine uniqueness of the second as it fits into a larger diagram.
This second type of uniqueness is already settled by Taylor:
the congruence classes of these diagrams
is a torsor over the cohomology group $H^2(\pi_0(G),\pi_1(G))$.
Similarly, the congruence classes of the second extensions
is a torsor over a quotient group 
$H^2(\pi_0(G),\pi_1(G))/\Delta(H^1(G_0,\pi_1(G))^{\pi_0(G)})$. 
That existence of a universal covering group is determined by a cohomology group one degree 
higher than the cohomology group determining uniqueness has parallels with previous work by two 
of the authors in extending representations from subgroups \cite{RWes}.

Our main accomplishment in the present paper is that we
relate the existence of the universal cover group $\widetilde{G}$ 
to the splitness 
of the topological 2-group $\widetilde{\sG}$
associated to the topological group $G$.
The 2-group $\widetilde{\sG}$
is the 2-group of the crossed module
$\widetilde{G_0}\rightarrow G$. 
It admits
a Sinh cocycle $\theta \in Z^3(\pi_0(G),\pi_1(G))$
that controls whether $\widetilde{\sG}$ is split, i.e., 
2-equivalent to a skeletal and strict 2-group.
Now we can state the main result of this paper
(see Theorem~\ref{main_th} for the full statement):
\begin{center}
{\em   The equality }  
  $[\theta]=[\dot\eta^\sharp]$
{\em   holds. Hence, the group $\widetilde{G}$ exists if and only if
  the 2-group $\widetilde{\sG}$ is split.}
\end{center}
It would be interesting to have a conceptual, non-computational proof
of this result. 

We start with a brief historic review in
Section~\ref{s0}, where we also contrast our approach with other
developments since Taylor's work. 



In Section~\ref{s1} we commence our study
by defining the topological 2-group $\widetilde{\sG}$.
The 3-cocycle $\theta\in Z^3(\pi_0(\widetilde{\sG}),\pi_1(\widetilde{\sG}))$
associated to a 2-group is a well-known construction \cite{BL04, Elg}
that is attributed to Sinh's Thesis \cite{Si}.
We give an explicit formula (Equation~(\ref{cocycle3})),
adapted to $\widetilde{\sG}$,
for $\theta\in Z^3(\pi_0(G),\pi_1(G))$
(note that $\pi_0(G)=\pi_0(\widetilde{\sG})$
and
$\pi_1(G)=\pi_1(\widetilde{\sG})$).

We develop essential algebraic tools in Section~\ref{s3}.
We require an obstruction for lifting a central extension to
an abelian extension along another extension.
This obstruction is known in the language
of abstract kernels and crossed resolutions.
In particular, we develop these tools in the context of
extensions of topological groups
and prove Theorem~\ref{main_th},
the main theorem of the present paper. We also derive an explicit formula for the
Taylor cocycle (Equation~(\ref{cocycle3}))
required for further use. 
The cocycle
$[\dot{\eta}^\sharp]\in H^3(\pi_0(G),\pi_1(G))$
controls existence of $\widetilde{G}$ as an abstract group.
We conclude in Theorem~\ref{main_th}
that such $\widetilde{G}$ also has the structure of a topological group and that this cocycle is an object we already know: it is precisely the obstruction $[\theta]$ from Section~\ref{s1}. 


Finally, in Section~\ref{s4} we generalize our results to the case
of a more general extension.




\section{Historical review}
\label{s0}

With a rich history of results and significant differences in terminology
we think that a historic review could benefit the reader. Perhaps the earliest results relevant to our interests are 
those of Mac Lane and Whitehead in the 1940s \cite{ML49,MW,Wh}. Using their idea of a crossed module, Taylor answers one of the questions considered in the present paper in 1954 \cite{Ta2}. He obtains his Theorem 6.5 which says that a topological group (with some topological conditions) has a universal covering group if and only if the obstruction of the crossed module $\widetilde{G_0}\to G$ vanishes; he also shows that the second cohomology group controls uniqueness \cite{Ta1}. His procedure makes ample use of the notion of extensions of crossed modules.

The next development worth mentioning is the 1976 paper by Brown and Spencer \cite{BS}, where they prove that the category of crossed modules is equivalent to the category of ``group-groupoids'' (better known today as ``strict 2-groups''). This result, they claim, was known to Verdier and Duskin a decade prior, but they are the first authors to publish a proof of it. They also extend this result to an equivalence of 2-categories. Furthermore, given a topological group $G$ one can obtain the strict 2-group $\sG$ -- see Section~\ref{s1} for more details. By Brown-Spencer's results, this strict 2-group corresponds to a crossed module, which hence has an obstruction in $H^3(\pi_0(G),\pi_1(G))$. 


Around the same time, Sinh 
shows in her thesis \cite{Si} 
that a coherent 2-group (called a \emph{gr-category} by Sinh) is determined up to equivalence by a group $H$, an abelian group $A$, an action of $H$ on $A$ by automorphisms, and the cohomology class of a 3-cocycle $[\theta]$ in $H^3(H,A)$. This correspondence was also explained by Joyal and Street in Sections 2 and 6 of \cite{JS}, a 1986 draft of the paper \cite{JS2} which would be published in 1993 without this explanation. A modern treatment of Sinh's results can be found in papers by Baez and Lauda \cite{BL04} and Elgueta \cite{Elg}, where the 3-cocycle is found using associators in the 2-group. This 3-cocycle vanishes precisely when the 2-group is strict and skeletal \cite{Elg}. The reader may consult Section~\ref{s1} of this paper for a self-contained explanation of this construction in the context of universal covers of topological groups.


The final development in our review is the 1994 paper by Brown and Mucuk \cite{BM}. They use the developments of the preceding 40 years in order to reinterpret and generalize Taylor's results through the lens of strict 2-groups and (using Brown and Spencer's results from \cite{BS}) crossed modules. The reader should note that while both Taylor and Brown-Mucuk use the notion of crossed modules in their approaches, they are nonetheless quite different -- in particular, Brown and Mucuk's approach avoids the lengthy algebraic exposition of Taylor's series of papers.

The reader will note that the role of strict 2-groups in this topic, while important, is a vehicle to turn the question into one about crossed modules. There are many good reasons for this, however we feel that this approach can miss the significance of topological 2-groups in answering the main question. The fundamental benefit of looking at topological 2-groups is that the universal covering group of a disconnected topological group (with a suitably nice topology) \emph{always exists} as a topological 2-group. Hence, the question at the heart of this study is: when does the universal covering topological \emph{2-group} give rise to a universal covering topological \emph{group}? In particular, what structure of the topological 2-group is required to get the appropriate group structure? We answer these questions in the following sections.

\section{2-Groups related to a topological group}
\label{s1}
Let $G$ be a locally arcwise connected,
semilocally simply-connected topological group
with the identity component $G_0$ and its universal cover
group $\widetilde{G_0}$. 
We follow the standard 2-group theoretic terminology \cite{BL04}, while
using the notation in Rumynin, Wendland \cite{RW}.
Let $\sG$ and $\widetilde{\sG}$
be the 2-groups associated to the crossed modules
$G_0\rightarrow G$ and $\widetilde{G_0}\rightarrow G$
correspondingly.
The natural map of crossed modules
$[\widetilde{G_0}\rightarrow G] \longrightarrow [G_0\rightarrow G]$
gives a homomorphism of topological  2-groups
$\widetilde{\sG}\rightarrow \sG$.

Let us examine $\widetilde{\sG}$,  the 2-group theoretic counterpart of the universal cover of $G$,
in greater detail. Recall that a 2-group is a 2-category with one 0-object
where all 1-morphisms are 1-isomorphisms and
all 2-morphisms are 2-isomorphisms.
Thus, the 0-objects are the 1-element set:
$\widetilde{\sG}_0=\{\star\}$.
Now $\widetilde{\sG}_1 (\star,\star)$ needs to be a (monoidal) category, while
$\widetilde{\sG}_2 (x,y)
= \widetilde{\sG}_1 (\star,\star) (x,y)$ 
is the set of morphisms between objects
$x,y\in \widetilde{\sG}_1 (\star,\star)$. For this 2-group we have 
$$
\widetilde{\sG}_1 (\star,\star) = G, \ \ 
\widetilde{\sG}_2 (x,y) = \{ \llbracket\gamma\rrbracket \,\mid \mbox{ homotopy class of continuous }
\gamma\colon [0,1]\rightarrow G, \, \gamma(0)=x, \, \gamma(1) = y \}.
$$
The horizontal composition and inverse
form a monoidal structure on the category
$\widetilde{\sG}_1 (\star,\star)$.
They come from the group operations on $G$:
$$
x\diamond y = xy, \ \llbracket \beta\rrbracket \diamond\llbracket \gamma\rrbracket =\llbracket \beta(\vt)\gamma(\vt)\rrbracket .
$$
In particular the 2-group $\widetilde{\sG}_1$
is {\em strict}: the associativity and the inverse property
hold on the nose.
The vertical composition and inverse are concatenation and reversal of paths:
$$
\beta\bullet \gamma (\vt)=
\xymatrix{
\star
   \ar@/^2pc/[rr]_{\quad}^{z}="3"
    \ar[rr]|{y}="2"
   \ar@/_2pc/[rr]_{x}="1"
&& \star
\ar@{}"1";"2"|(.2){\,}="7"
\ar@{}"1";"2"|(.8){\,}="8"
\ar@{=>}"7" ;"8"^{\beta}
\ar@{}"2";"3"|(.2){\,}="9"
\ar@{}"2";"3"|(.8){\,}="10"
\ar@{=>}"9" ;"10"^{\gamma}
}
=
\begin{cases} 
  \beta (2\vt) \  & \ \mbox{ if }  \ \vt\leq 1/2,  \\
  \gamma (2\vt -1) \  & \ \mbox{ if }  \ \vt\geq 1/2 , 
\end{cases}
\ \ \ \ \ \ \ \ 
\gamma^{-1\bullet} (\vt) = \gamma (1-\vt).
$$
The 2-group $\sG$ admits a similar, somewhat easier description
with the same horizontal operations and the trivial vertical operations:
$$
{\sG}_0 = \{\star\}, \  \ 
{\sG}_1 (\star,\star) = G, \ \ 
{\sG}_2 (x,y) =
\begin{cases} 
  \{\sI_{x,y} \} \  & \ \mbox{ if }  \ y^{-1}x\in G_0,  \\
  \  \ \emptyset \  & \ \mbox{ if }  \ y^{-1}x\not\in G_0 . 
\end{cases}
$$
Both $\widetilde{\sG}$ and $\sG$
are topological 2-groups because
$\widetilde{\sG}_1=\sG_1=G$,  $\widetilde{\sG}_2$
and $\sG_2$
inherit topologies from $G$
under which all the operations
are continuous. 
The 2-group homotopic properties of $\widetilde{\sG}$
resemble those of $G$.
Recall that $\pi_0(\widetilde{\sG})$
is the group of isomorphism classes in $\widetilde{\sG}_{1}(\star,\star)$ \cite{Elg}. 
Clearly, both $\pi_0(\widetilde{\sG})$ and $\pi_0(\sG)$
are naturally isomorphic to $\pi_0(G)$.
Also recall that
$\pi_1(\widetilde{\sG}) =  \widetilde{\sG}_{2}(\vi_\star,\vi_\star)$ \cite{Elg}. 
Again, it is clear that
$\pi_1(\widetilde{\sG})=\widetilde{\sG}_{2}(1,1)$ is naturally isomorphic to $\pi_1(G)$,
while
$\pi_1(\sG)= {\sG}_{2}(\vi_\star,\vi_\star) ={\sG}_{2}(1,1) =\{\sI_{1,1}\}$
is the trivial group.

This gives a 2-group-theoretic action of $\pi_0(\widetilde{\sG})=\pi_0(G)$
on $\pi_1(\widetilde{\sG})=\pi_1 (G)$ \cite{Elg} (cf. \cite{BL04}):
$$
\llbracket g\rrbracket \cdot x \coloneqq g \diamond x \diamond g^{-1}.
$$

Let us now recall the standard action.
Think of the universal cover as the end-preserving
homotopy classes
of continuous paths:
$$
\widetilde{G_0}
= \{ \llbracket \gamma\rrbracket  \,\mid \mbox{ homotopy class of continuous }
\gamma\colon [0,1]\rightarrow G, \, \gamma(0)=1_G \}
$$
with the pointwise multiplication and inverses
$$
\llbracket \beta\rrbracket \llbracket \gamma\rrbracket  = \llbracket \beta(\vt)\gamma(\vt)\rrbracket , \
\llbracket \beta\rrbracket ^{-1} = \llbracket \beta(\vt)^{-1}\rrbracket .
$$
The map $\widetilde{G_0} \rightarrow G_0$ is given by
$\llbracket \gamma\rrbracket \mapsto \gamma (1)$
so that $\pi_1(G)= \{ \llbracket \gamma\rrbracket  \,\mid \gamma:\, \gamma(0)=\gamma(1)= 1_G \}$.

A set theoretic splitting $\alpha\colon\pi_0 (G)\rightarrow G$
is tantamount to the choice of an element on each connected component: 
$\alpha (gG_0) = \overline{g}$.
Since $\pi_1(G)$ is a central subgroup of $\widetilde{G_0}$,
it becomes a $G$-module, trivial on $G_0$:
$$
\,^g \llbracket \gamma\rrbracket  = \llbracket g\gamma(\vt)g^{-1}\rrbracket .
$$
In particular, it is a $\pi_0(G)$-module.
This is the standard action. It is the same as the 2-group-theoretic action.

We have already discussed what it means for a 2-group $\sH$ (with
$\sH_0=\{\star\}$) to be {\em strict}. All the algebraic properties
of the monoidal category $\sH_1 (\star,\star)$ must hold on the nose:
$$
(x\diamond y) \diamond z = x\diamond (y \diamond z), \ \ \ 
x\diamond \vi_\star = \vi_\star \diamond x = x, \ \ \ 
x\diamond x^{-1}=  \vi_\star = x^{-1} \diamond x
$$
for all objects $x,y,z\in \sH_1 (\star,\star)$.
A 2-group $\sH$ is called {\em skeletal}, if all 2-morphisms
are automorphisms, i.e.,
$$
\sH_2 (x,y) \neq \emptyset
\ \ \ \mbox{ if and only if } \  \ 
x = y.
$$
A 2-group is called {\em split}, if it is 2-equivalent to a strict skeletal 2-group. 
The theory of 2-groups gives us a Sinh cocycle
$\theta_\sH \in Z^3(\pi_0(\sH),\pi_1(\sH))$
associated to the 2-group ${\sH}$ \cite{BL04,Elg}.
This is the obstruction to splitness: $[\theta_\sH]=0$ if and only if
$\sH$ is split. 

Instead of the general construction
we recall briefly how to build a Sinh cocycle
$\theta = \theta_{\widetilde{\sG}} \in Z^3(\pi_0(G),\pi_1(G))$
for the 2-group $\widetilde{\sG}$ of our primary interest.
Let us try to build a skeletal 2-group $\sS$, 2-equivalent to
$\widetilde{\sG}$. Since $\sS$ and $\widetilde{\sG}$
have the same homotopic properties, inevitably we must have
$$
\sS_0=\{\star\}, \  
\sS_1 (\star,\star) = \pi_0(G), \ 
\sS_2 (g,h) =
\begin{cases} 
  \pi_1 (G, \overline{g}) \  & \ \mbox{ if }  \ g=h,  \\
        \emptyset, \  & \ \mbox{ if } \ g\neq h, 
\end{cases}
$$
where $\overline{g}\in G$ is some point lifting $g\in\pi_0 (G)$,
i.e., $\overline{g}G_0 = g$. 
Now the vertical composition $\bullet$ is the concatenation of paths
(or just the multiplication in  $\pi_1 (G, \overline{g})$).
To assemble the standard 2-equivalences
$$
\widetilde{\sG}\Rightarrow\sS, \ x \mapsto xG_0; \ \ 
\sS\Rightarrow\widetilde{\sG}, \ g \mapsto \overline{g}, 
$$
the horizontal composition $\diamond$ in $\sS$
must come from the pointwise multiplication in $G$, but 
there is an issue:
$\sS_2 (g,g) \diamond \sS_2 (h,h) = \pi_1 ( G, \overline{g}\,\overline{h})$,
not 
$\sS_2 (gh,gh) = \pi_1 ( G, \overline{gh})$ as we wish for.
We can identify paths using right multiplications:
$$
R_x \colon \pi_1(G) \xrightarrow{\cong} \pi_1 (G, x) , \ \  \big(R_x (\gamma)\big) (\vt) = \gamma(\vt)x
$$
but it is not functorial. 
To make things work
we need to choose a path $\beta_x$
from $1$ to each $x\in G_0$.
This gives a path 
$\beta_{x,y}\coloneqq R_x(\beta_{yx^{-1}})$ from $x$ to $y$
for every pair of elements from the same component, i.e.,
for all $x\in G_0 y$.
We define the horizontal composition on morphisms
using $\phi_{f,g}\coloneqq \beta_{\overline{fg},\overline{f}\,\overline{g}}\ $ : 
$$
\sS_2 (g,g) \times \sS_2 (h,h) \xrightarrow{\diamond}
\sS_2 (gh,gh), \ \
   \llbracket \gamma\rrbracket  \diamond \llbracket \delta\rrbracket  =
\phi_{g,h}
   \bullet
       \llbracket \gamma \delta\rrbracket 
       \bullet
\phi_{g,h}^{\; -1 \bullet}
       \ .$$
       Being skeletal comes at a cost: the new 2-group $\sS$
       is no longer strict, in general (one may similarly observe this imbalance in the functor $\Sigma$ appearing in \cite[Proposition 9 et al.]{JS}). 
       The associativity
       constraint in $\sS$ is a natural isomorphism of trifunctors
       $
       \Ass \colon  
       ( \; \underline{\ \ } \diamond  \underline{\ \ } \; )\; \diamond
       \underline{\ \ }
       \; \rightarrow \; 
       \underline{\ \ } \; \diamond (\; \underline{\ \ } \diamond  \underline{\ \ } \; )
       \; 
       $, given by
\begin{equation}\label{ass}
       \pi_1(G, \overline{fgh})
\ni 
\Ass_{f,g,h} =
\llbracket
\phi_{f,gh}\bullet
              (\vi_f \diamond \phi_{g,h})
       \bullet
              (\phi_{f,g}\diamond \vi_h)^{-1 \bullet}
\bullet
       \phi_{fg,h}^{\;-1 \bullet} \rrbracket \; .
\end{equation}
An interested reader can verify the pentagon condition.
The cocycle $\theta$ is obtained from the associativity
constraint by moving the base point to the identity element:
       $$
       \pi_1(G,1) \ni 
       \theta (f,g,h) 
       = R_{\overline{fgh}^{\; -1}} (\Ass_{f,g,h})
       = R_{\overline{fgh}}^{\; -1} \, (\Ass_{f,g,h}).  
       $$
       The cocycle property for $\theta$ follows from the pentagon condition for
       the associator $\Ass$.
       Let us define
       $\eta(f,g)\coloneqq \overline{f}\,\overline{g}\,(\overline{fg})^{-1}\in G_0$ for all $f,g\in\pi_0(G)$.
       Observe that
       $$
       R_{\overline{fg}}^{\, -1} \, (\phi_{f,g})
       =
       R_{\overline{fg}}^{\, -1} \, (\beta_{\overline{fg},\overline{f}\,\overline{g}})
       =
       R_{\overline{fg}}^{\, -1} \, (  R_{\overline{fg}} \,
       (\beta_{\eta (f,g)}))
       = \beta_{\eta (f,g)}\, ,
       $$
      \begin{multline*}
       R_{\overline{fgh}}^{-1} \, (\vi_f\diamond \phi_{g,h})
       =R_{\overline{fgh}}^{-1} \, (\llbracket \vi_f(\vt)\phi_{g,h}(\vt)\rrbracket )
       = R_{\overline{fgh}}^{-1} \, (L_{\overline{f}}R_{\overline{gh}}(\beta_{\eta(g,h)}))=\\=
       R_{\overline{fgh}}^{-1} \, (R_{\overline{f}\,\overline{gh}}(\,^{\overline{f}}\beta_{\eta(g,h)})) 
       =R_{\overline{fgh}}^{-1} \, (R_{\overline{fgh}}R_{\eta(f,gh)}(\,^{\overline{f}}\beta_{\eta(g,h)}))
       = R_{\eta(f,gh)}\,^{\overline{f}}\beta_{\eta(g,h)} \, , 
       \end{multline*}
       \begin{multline*}
       R_{\overline{fgh}}^{-1} \, (\phi_{f,g} \diamond \vi_h )
       = R_{\overline{fgh}}^{-1}\, (\llbracket \phi_{f,g}(\vt)\vi_h(\vt)\rrbracket )
       = R_{\overline{fgh}}^{-1}\, (R_{\overline{h}}R_{\overline{fg}}(\beta_{\eta(f,g)}))
       	\\ = R_{\overline{fgh}}^{-1}\, (R_{\overline{fg}\,\overline{h}}(\beta_{\eta(f,g)}))
       = R_{\overline{fgh}}^{-1}\, (R_{\overline{fgh}}R_{\eta(fg,h)}(\beta_{\eta(f,g)}))		
       = R_{\eta(fg,h)}(\beta_{\eta(f,g)}) \, .
       \end{multline*}
       so that Equation~(\ref{ass}) gets translated into
       an explicit formula for $\theta$: 
       \begin{equation}\label{cocycle}
         \theta(f,g,h) =
         \llbracket \beta_{\eta(f,gh)}\bullet R_{\eta(f,gh)}\,^{\overline{f}}\beta_{\eta(g,h)}\bullet R_{\eta(fg,h)}(\beta_{\eta(f,g)})^{-1{\bullet}}\bullet\beta_{\eta(fg,h)}^{-1{\bullet}} \rrbracket \, .  
         \end{equation}

%

\section{Obstruction for existence of the universal cover}
\label{s3}
We continue with a locally arcwise connected,
semilocally simply-connected topological group $G$
and its identity component $G_0$.
A universal covering group for $G$ should fit into the following diagram:

\begin{equation}
\label{diag1}
\begin{CD}
@. 1 @. 1@.  @.  \\  
@. @VVV @VVV @. @. \\  
@. \pi_1(G) @. \pi_1(G) @. @. \\
@. @VVV @VV{i}V @. @. \\
@. \widetilde{G_0} @. \widetilde{G}@. @.\\
@. @VVV @VV{j}V @. @. \\
1 @>>> G_0 @>>> G@>>> \pi_0(G) @>>> 1 \\
@. @VVV @VVV @. @. \\  
@. 1 @. 1@.  @.  \\  
\end{CD}
\end{equation}

Observe that the group $G$ acts on $\widetilde{G_0}$ by the pointwise conjugation of paths:
$
g\cdot [\gamma (\vt) ] = [g\gamma(\vt)g^{-1}].
$
In particular, this induces a $G$-module structure on $\pi_1(G)$ which extends the trivial $G_0$-module structure.
The reader should also note that the bottom row is an exact sequence and the left column is a central extension. In general, 
one should expect the middle column to be an abelian extension, but not necessarily a central one.

The algebraic question of when there exist an abstract group $\widetilde{G}$ and 
group homomorphisms $i$ and $j$ that fit in Diagram~(\ref{diag1}) has been tackled by various authors.
See, for example, the work of Ratcliffe \cite{Rat}, Huebschmann \cite{Hue4} and Wu \cite{Wu}, or Huebschmann \cite{Hue3} for an approach in a more generalized setting. Their main procedure is as follows: let 
$$1\to N\to H\to P\to 1$$ be an exact sequence of groups, and let $A$ be an $H$-module on which $N$ acts trivially. Denote by $\Opext_H(N,A)$ the abelian group consisting of equivalence classes of extensions of $N$ by $A$ which are equipped with an $H$-action respecting the $H$-actions on $A$ and $N$ ($H$ acts on $N$ by conjugation). For the precise definition of the equivalence relation on extensions and the group structure of $\Opext_H(N,A)$, the reader should consult the above sources. This group fits into the exact sequence
\begin{equation}\label{Opext}
H^2(P,A)\to H^2(H,A)\xrightarrow{\gamma} \Opext_H(N,A)\xrightarrow{\epsilon} H^3(P,A)\to H^3(H,A).
\end{equation}

In the situation at hand, we start with an element $\Psi\in \Opext_G(G_0,\pi_1(G))$ corresponding to the left column of Diagram~(\ref{diag1}), and we 
wish to determine the existence of an abelian extension $$1\to \pi_1(G) \to \widetilde{G}\to G\to 1$$ which completes the diagram. In other words, we require $\Psi$ to lie in the image of the homomorphism $\gamma:H^2(G,\pi_1(G))\to \Opext_G(G_0,\pi_1(G))$, or, equivalently, that $\Psi$ lies in the kernel of the homomorphism $\epsilon:\Opext_G(G_0,\pi_1(G))\to H^3(\pi_0(G),\pi_1(G))$. We will abuse notation to use $\Psi$ both for the specific central extension in the above diagram, and its equivalence class.

Ratcliffe's work in \cite{Rat} gives the image of $\Psi$ under this map explicitly. Specifically, Diagram~\ref{diag1} gives a crossed 2-fold extension $$1\to \pi_1(G)\to \widetilde{G_0}\xrightarrow{\rho} G\to \pi_0(G),$$ and the map $\epsilon$ sends $\Psi$ to the Eilenberg-Mac Lane obstruction of this extension. This can be constructed as follows (see \cite{MW}). Recall that we let $\overline{f}\in G$ be a fixed lifting of $f\in\pi_0(G)$. Define $$\eta:\pi_0(G)\times \pi_0(G)\to G, \qquad\, \eta(f,g)=\overline{f} \, \overline{g} \, (\overline{fg})^{-1}.$$ Since $\pi_0(G)$ is abelian, $\eta(f,g)$ is in the kernel of the projection $G\to \pi_0(G)$. Hence, by exactness there exists $\dot{\eta}:\pi_0(G)\times\pi_0(G)\to \widetilde{G}$ such that $\rho(\dot{\eta}(f,g))=\eta(f,g)$ for all $f,g\in\pi_0(G)$. We may assume $\dot{\eta}(f,1)=\dot{\eta}(1,g)=1$ for all $f,g\in\pi_0(G)$.

Now, $d\dot{\eta}(f,g,h)=\,^{\overline{f}}\dot{\eta}(g,h)\dot{\eta}(f,gh)(\dot{\eta}(f,g)\dot{\eta}(fg,h))^{-1}$ maps to the identity under $\rho$, and hence lies inside $\pi_1(G)$. Letting $\dot{\eta}^{\sharp}\coloneqq d\dot{\eta}:\pi_0(G)\times \pi_0(G)\times \pi_0(G)\to \pi_1(G)$, we get that $[\dot{\eta}^{\sharp}]\in H^3(\pi_0(G),\pi_1(G))$ is the Eilenberg-Mac Lane obstruction of the above extension. It can be shown that this procedure is independent of the lifting of $\pi_0(G)$ to $G$. As a result, the homomorphism $\epsilon:\Opext_G(G_0,\pi_1(G))\to H^3(\pi_0(G),\pi_1(G))$ sends $\Psi$ to $[\dot{\eta}^{\sharp}]\in H^3(\pi_0(G),\pi_1(G))$. We shall call $\dot{\eta}^{\sharp}$ the \emph{Taylor cocycle} for historic reasons.

The reader should note that the crossed module $\widetilde{G_0}\to G$ we consider here corresponds to the strict 2-group $\widetilde{\sG}$ examined in Section \ref{s1} under Brown-Spencer's correspondence \cite{BS}. In particular, the following theorem shows that the Sinh cocycle of the 2-group $\widetilde{\sG}$ is cohomologous to the Eilenberg-Mac Lane obstruction of its corresponding crossed module $\widetilde{G_0}\to G$.

\begin{theorem}
	\label{main_th}
	Let $G$ be a
	locally arcwise connected,
	semilocally simply-connected topological group.
	The following statements related to objects $G_0$, $\widetilde{G_0}$, $\pi_1(G)$, $\pi_0(G)$, $\dot{\eta}^\sharp$ and $\theta$, unveiled
	in the preceding passage, hold:
	\begin{enumerate}
		\item The universal cover group $\widetilde{G}$ exists if and only if
		$[\dot{\eta}^\sharp]=0 \in H^3(\pi_0(G),\pi_1(G))$.
		\item The classes are equal: 
		$[\dot{\eta}^\sharp]= [\theta] \in H^3(\pi_0(G),\pi_1(G))$.
		\item The universal cover topological group $\widetilde{G}$ exists if and only if
		the 2-group $\widetilde{\sG}$ is 2-equivalent to a skeletal strict 2-group.
	\end{enumerate}
\end{theorem}
\begin{proof}
	(1) The existence of the universal cover group as an abstract group is determined by the vanishing of the image of $\Psi\in\Opext_G(G_0,\pi_1(G))$ under the map  $\epsilon:\Opext_G(G_0,\pi_1(G))\to H^3(\pi_0(G),\pi_1(G))$. This is precisely the condition that $[\dot{\eta}^\sharp]=0 \in H^3(\pi_0(G),\pi_1(G))$. 
	
	All that remains is to show that the abstract group $\widetilde{G}$ is a topological group. Once the universal cover group exists as an abstract group, we observe that Diagram~(\ref{diag1}) can be extended to a commutative diagram of the following form.
	\begin{equation}
	\label{diag2}
	\begin{CD}
	@. 1 @. 1@.  @.  \\  
	@. @VVV @VVV @. @. \\  
	@. \pi_1(G) @= \pi_1 (G)@. @.\\
	@. @VVV @VV{i}V @. @.\\
	1 @>>>\widetilde{G_0} @>{k}>> \widetilde{G}@>>> \pi_0(G) @>>> 1\\
	@. @VVV @VV{j}V @| @. \\
	1 @>>> G_0 @>>> G@>>> \pi_0(G) @>>> 1\\
	@. @VVV @VVV @. @. \\  
	@. 1 @. 1@.  @.  \\
	\end{CD}
	\end{equation}
	
	This follows from the facts that, set-wise, $\widetilde{G_0}=\pi_1(G)\times G_0$ and $\widetilde{G}=\pi_1(G)\times G$, and that $\gamma:H^2(G,\pi_1(G))\to \Opext_G(G_0,\pi_1(G))$ induces the restriction map $H^2(G,\pi_1(G))\to H^2(G_0,\pi_1(G))$ (see \cite{Rat}). Together, these two observations mean that the natural inclusion $\widetilde{G_0} \hookrightarrow \widetilde{G}$ is a group homomorphism, with quotient group $\pi_0(G)$.
	
	 This observation shows that $\widetilde{G}$ is an extension of $\pi_0(G)$ by $\widetilde{G_0}$, so is constructed from a 2-cocycle $\mu\in Z^2(\pi_0(G),\widetilde{G_0})$. In particular, it is equal as a set to $\widetilde{G_0}\times\pi_0(G)$, with a group structure given by:
	$$
	(g,p) \ast (g',p') = (g \alpha_p(g') \mu (p,p'), pp'),\qquad (g,p)^{-1\ast}=(\alpha_p^{-1}(g')^{-1}\;\alpha_p^{-1}(\mu(p,p^{-1}))^{-1},p^{-1}).
	$$
	Here $\alpha_p\in\Aut(\widetilde{G_0})$ with $\alpha_p(g)\coloneqq \,^{\overline{p}}g$.	Both the multiplication and the inverse map are continuous by inspection.
	
	(2) Recall that $\eta(f,g)=\overline{f} \, \overline{g} \, (\overline{fg})^{-1}$ for $f,g\in\pi_0(G)$, where $\overline{f} \in{G} $ is a fixed lifting of $f \in \pi_0(G)$.
	
	Let us compute $\theta(f,g,h)^{-1\bullet}$ for $f,g,h\in\pi_0(G)$. Recall that for $g \in G_0$, $\beta_g$ is a path from $1_G$ to $g$ -- this corresponds to a lifting from $G_0$ to $\widetilde{G_0}$. 
	From formula (\ref{cocycle}) we get:   
	$$\theta(f,g,h) = \llbracket \beta_{\eta(f,gh)}
	\bullet R_{\eta(f,gh)}\,^{\overline{f}}\beta_{\eta(g,h)}
	\bullet R_{\eta(fg,h)}\beta_{\eta(f,g)}^{-1{\bullet}}
	\bullet \beta_{\eta(fg,h)}^{-1{\bullet}} \rrbracket.$$
	
	For ease of notation, we set $\beta_1=\beta_{\eta(f,gh)}$, $\beta_2=\,^{\overline{f}}\beta_{\eta(g,h)}$, $\beta_3=\beta_{\eta(f,g)}$ and $\beta_4=\beta_{\eta(fg,h)}$, so that we have  
	$$\theta(f,g,h)=\llbracket \beta_1 \bullet  R_{\eta(f,gh)}(\beta_2)\bullet R_{\eta(fg,h)}({\beta_3})^{-1\bullet}\bullet \beta_4 ^{-1\bullet} \rrbracket.$$
	
	On the other hand, by definition $\beta_1$ is a lifting of $\eta(f,gh)$ and hence we have $\beta_1=\dot{\eta}(f,gh)$ in the above notation, and similarly for $\beta_2$, $\beta_3$, $\beta_4$. This allows us to compute $\dot{\eta}^\sharp(f,g,h)$ for $f,g,h\in\pi_0(G)$. 
	\begin{equation}
	\label{cocycle3}
	\dot{\eta}^\sharp(f,g,h)=\llbracket \beta_2\rrbracket \diamond \llbracket \beta_1\rrbracket \diamond \llbracket \beta_4\rrbracket^{-1\diamond} \diamond \llbracket \beta_3\rrbracket^{-1\diamond}=\llbracket \beta_2(\vt)\beta_1(\vt)\beta_4(\vt)^{-1}\beta_3(\vt)^{-1}\rrbracket \; .
	\end{equation}
	
	Since $\dot{\eta}^\sharp(f,g,h)$ is central in $\widetilde{G_0}$, we also have $$\dot{\eta}^\sharp(f,g,h)=\llbracket \beta_4\rrbracket^{-1\diamond}\diamond\llbracket \beta_3\rrbracket^{-1\diamond}\diamond\llbracket \beta_2\rrbracket\diamond\llbracket \beta_1\rrbracket=\llbracket \beta_4(\vt)^{-1}\beta_3(\vt)^{-1}\beta_2(\vt)\beta_1(\vt)\rrbracket.$$ Observe that $\dot{\eta}^\sharp(f,g,h)=\llbracket \beta_4(F_4(\vt))^{-1}\beta_3(F_3(\vt))^{-1}\beta_2(F_2(\vt))\beta_1(F_1(\vt))\rrbracket$ for continuous maps $F_1,F_2,F_3,F_4\colon [0,1] \rightarrow [0,1]$. Using the simply-connectedness of $[0,1]$, we get that $\beta_4(F_4(\vt))^{-1}\beta_3(F_3(\vt))^{-1}\beta_2(F_2(\vt))\beta_1(F_1(\vt))$ is homotopic to the path
	
	$$\lambda(\vt)=\fourpartdef{\beta_4(0)^{-1}\beta_3(0)^{-1}\beta_2(0)\beta_1(4\vt)=\beta_1(4\vt)}{0\leq \vt\leq 1/4}{\beta_4(0)^{-1}\beta_3(0)^{-1}\beta_2(4\vt-1)\beta_1(1)=\beta_2(4\vt-1)\eta(f,gh)}{1/4\leq \vt\leq 1/2}
	{\beta_4(0)^{-1}\beta_3(4\vt-2)^{-1}\beta_2(1)\beta_1(1)=\beta_3(4\vt-2)^{-1}\,^{\overline{f}}\eta(g,h)\eta(f,gh)}{1/2\leq \vt\leq 3/4}
	{\beta_4(4\vt-3)^{-1}\beta_3(1)^{-1}\beta_2(1)\beta_1(1)=\beta_4(4\vt-3)^{-1}\eta(f,g)^{-1}\,^{\overline{f}}\eta(g,h)\eta(f,gh)}{3/4\leq \vt\leq 1}$$
	
	Now we just have to show that
	$$\beta_3(0)^{-1}\beta_3(1-\vt)\eta(fg,h)={\beta_3}^{-1\bullet}(\vt)\eta(fg,h)\sim \beta_3(\vt)^{-1}\,^{\overline{f}}\eta(g,h)\eta(f,gh)=\beta_3(\vt)^{-1}\beta_3(1)\eta(fg,h),$$
	$$\beta_4(0)^{-1}\beta_4(1-\vt)=\beta_4^{-1\bullet}(\vt)\sim\beta_4(\vt)^{-1}\eta(f,g)^{-1}\,^{\overline{f}}\eta(g,h)\eta(f,gh)=\beta_4(\vt)^{-1}\beta_4(1).$$
	The first line follows from the homotopy
	$F(\vs,\vt)=\beta_3((1-\vs)\vt)^{-1}\beta_3(1-\vs\vt)\eta(fg,h)$
	and the second line from the homotopy
	$F(\vs,\vt)=\beta_4((1-\vs)\vt)^{-1}\beta_4(1-\vs\vt)$.
	
	(3) The first statement is equivalent to $[\dot{\eta}^\sharp]=0$ by (1).
	The second statement is equivalent to $[\theta]=0$ \cite{BL04,Elg}.
	Thus, everything follows from part (2).
\end{proof}

For the uniqueness of such a universal cover group, one needs to examine the beginning of the exact sequence~(\ref{Opext}). In fact, this exact sequence can be extended on the left to
\begin{equation}\label{Opext2}
H^1(H,A)\to H^1(N,A)^P\xrightarrow{\Delta} H^2(P,A)\to H^2(H,A)\xrightarrow{\gamma} \Opext_H(N,A)\to\ldots.
\end{equation}
See \cite{Hue3} for details, including a description of the map $\Delta:H^1(N,A)^P\to H^2(P,A)$.

 In particular, the group $H^2(\pi_0(G),\pi_1(G))$ plays a key role in questions of uniqueness. Note that we only consider uniqueness of the universal covering group up to congruence of the extension of $G$ by $\pi_1(G)$.

\begin{theorem}
	Let $G$ be a
	locally arcwise connected,
	semilocally simply-connected topological group. Let $\Psi\in \Opext_G(G_0,\pi_1(G))$ be the central extension of $G_0$ by $\pi_1(G)$ in Diagram~(\ref{diag1}), and let $E_\Psi\coloneqq \gamma^{-1}(\Psi)$.
	The following statements related to objects $G_0$, $\widetilde{G_0}$, $\pi_1(G)$ and $\pi_0(G)$ hold:
	\begin{enumerate}
		\item The abelian group $H^2(\pi_0(G),\pi_1(G))$ acts transitively on $E_\Psi$, with kernel $\Delta(H^1(G_0,\pi_1(G))^{\pi_0(G)})$.
		\item $H^2(\pi_0(G),\pi_1(G))/\Delta(H^1(G_0,\pi_1(G))^{\pi_0(G)})$ acts freely on $E_\Psi$.
		\item $H^2(\pi_0(G),\pi_1(G))/\Delta(H^1(G_0,\pi_1(G))^{\pi_0(G)})=0$ if and only if $\lvert E_\Psi\rvert \leq 1$.
	\end{enumerate}
\end{theorem}

\begin{proof}
	(1) Let $\xi:H^2(\pi_0(G),\pi_1(G))\to H^2(G,\pi_1(G))$ be the homomorphism in Diagram~(\ref{Opext}). For $[\tau]\in H^2(\pi_0(G),\pi_1(G))$ and $[\mu]\in E_\Psi$, define the action $[\tau]\cdot[\mu]=[\mu]+\xi([\tau])\in E_\Psi$. This is well-defined by the exactness of (\ref{Opext}), and it is straightforward to see that it is an action. The transitivity also follows from the exactness of (\ref{Opext}), since any two elements in $E_\Psi$ differ by an element of the kernel of $\gamma$.
	
	For the kernel of the action, it is clear that $[\tau]\in H^2(\pi_0(G),\pi_1(G))$ acts trivially on $[\mu]\in E_\Psi$ if and only if $\xi([\tau])=0$. This is precisely the requirement that $[\tau]\in \Delta(H^1(G_0,\pi_1(G))^{\pi_0(G)})$.
	
	(2) Follows from the proof of (1).
	
	(3) Follows easily from (1) and (2).
	
\end{proof}

%
%
%

We may also consider uniqueness of diagrams of the form of Diagram~(\ref{diag2}). We say here that two such diagrams are congruent if the associated abelian extensions 
$$1\to\pi_1(G)\to \widetilde{G}\to G\to 1\quad\mbox{and}\quad 1\to\pi_1(G)\to \widetilde{G}'\to G\to 1$$
are congruent though a map which respects the other arrows in the diagram. If two diagrams come from different elements of $E_\Psi$, they are clearly  non-congruent. However, even if they come from the same element of $E_\Psi$, there may still be homomorphisms $\widetilde{G_0} \hookrightarrow \widetilde{G}$ which give non-congruent diagrams. With this additional structure on universal covering groups of $G$, we get the following result.

\begin{prop}\label{corresp}
	There is a one-to-one correspondence between equivalence classes of universal covering groups of $G$ which fit into Diagram~(\ref{diag2}), and the cohomology group $H^2(\pi_0(G),\pi_1(G))$.
\end{prop}

\begin{proof}
	This result is Theorem 5.4 in \cite{BM}, or Theorem 6.5(b) in \cite{Ta2}.
\end{proof}

\section{More general covers}
\label{s4}
%
%
Let us consider an arbitrary cover group
$\widehat{G_0}\rightarrow G_0$. Since it factors via the universal cover,
its kernel is some quotient $\overline{\pi_1(G)} = \pi_1(G)/B$
of the fundamental group. Let us assume that $B$ is a 
$\pi_0(G)$-submodule of $\pi_1(G)$: otherwise, an extension of the cover
to $G$ is impossible.
An attempt to extend this cover to the whole
of the group $G$
yields a diagram, similar to 
Diagram~(\ref{diag1}):

\begin{equation}
\label{diag3}
\begin{CD}  
\overline{\pi_1(G)} @. \overline{\pi_1(G)} @. \\
@VVV @VV{i}V @. \\
\widehat{G_0} @. \widehat{G}@.\\
@VVV @VV{j}V @.\\
G_0 @>>> G@>>> \pi_0(G)\\
\end{CD}
\end{equation}

The requirement to have a topology on $\widehat{G}$
does not hinder any considerations: $j$ is a local homeomorphism,
fixing the topology. 

The crossed module $\widehat{G_0}\to G$ yields the Taylor cocycle
${\dot\eta}^\sharp\in Z^3(\pi_0(G),\overline{\pi_1(G)})$, as in Section~\ref{s3}.  

Also, we get a strict 2-group $\widehat{\sG}$ associated to
the crossed module $\widehat{G_0} \rightarrow G$:
$$
\widehat{\sG}_0=\{\star\}, \ \ 
\widehat{\sG}_1 (\star,\star) = G, \ \ 
\widehat{\sG}_2 (x,y) =
\{
\xymatrix{
	\star
	\ar@/^1pc/[rr]_{\quad}^{y}="2"
	\ar@/_1pc/[rr]_{x}="1"
	&& \star
	\ar@{}"1";"2"|(.2){\,}="7"
	\ar@{}"1";"2"|(.8){\,}="8"
	\ar@{=>}"7" ;"8"^{g}
} 
\,\mid\
g\in \widehat{G_0},\ y = \partial (g) x \}
.$$
This 2-group yields the Sinh cocycle
$\theta\in Z^3(\pi_0(G),\overline{\pi_1(G)})$.  

The 2-group $\widehat{\sG}$ also admits a homotopic description:
$\widehat{\sG}_2 (x,y)$ consist of those paths that are right translates
of paths comprising $\widehat{G_0}$. All our proofs go through.
Hence, we have the following corollary of the proof of 
Theorem~\ref{main_th}:
\begin{cor}
	\label{main_th_general}
	The following statements related to
	objects $G$, $G_0$, $\widehat{G_0}$, $\overline{\pi_1(G)}$, $\pi_0(G)$, $\dot{\eta}^\sharp$ and $\theta$, unveiled
	in the preceding passage, hold:
	\begin{enumerate}
	\item The topological cover group $\widetilde{G}$ exists if and only if
	$[\dot{\eta}^\sharp]=0 \in H^3(\pi_0(G),\overline{\pi_1(G)})$.
	\item The classes are equal: 
	$[\dot{\eta}^\sharp]= [\theta] \in H^3(\pi_0(G),\overline{\pi_1(G)})$.
	\item The topological cover group $\widehat{G}$ exists if and only if
	the 2-group $\widehat{\sG}$ is 2-equivalent to a skeletal strict 2-group.
	\item The abelian group $H^2(\pi_0(G),\overline{\pi_1(G)})/\Delta(H^1(G_0,\overline{\pi_1(G)})^{\pi_0(G)})$ acts freely and transitively on $E_\Psi$, with $\Delta$ and $E_\Psi$ defined as in Section~\ref{s3}.
	\item $H^2(\pi_0(G),\overline{\pi_1(G)})/\Delta(H^1(G_0,\overline{\pi_1(G)})^{\pi_0(G)})=0$ if and only if $\lvert E_\Psi\rvert \leq 1$.
	\end{enumerate}
\end{cor}

One can also obtain an analogue of Proposition~\ref{corresp} in this situation.

\end{document}